\documentclass[11pt]{article}              
\usepackage[T1]{fontenc}
\usepackage{amsmath,amsthm,amssymb}
\usepackage{txfonts}

\usepackage{bbm}
\usepackage{tikz}
\usetikzlibrary{automata,arrows,positioning,calc}
\usepackage{graphicx}

\newcommand{\R}{\mathbbm{R}}      

\newcommand{\Bc}{\mathcal{B}}
\newcommand{\Eb}{\mathbbm{E}}
\newcommand{\Fc}{\mathcal{F}}
\newcommand{\Gc}{\mathcal{G}}
\newcommand{\Hc}{\mathcal{H}}
\newcommand{\Lc}{\mathcal{L}}

\newcommand{\Pc}{\mathcal{P}}

\newcommand{\Uc}{\mathcal{U}}

\newcommand{\Vc}{\mathcal{V}}
\newcommand{\Xc}{\mathcal{X}}

\newcommand{\Zc}{\mathcal{Z}}
\newcommand{\1}{\mathbbm{1}}

\newcommand{\argmin}{\mathop{\rm argmin}}
\newcommand{\graph}{\mathop{\rm graph}}

\newcommand{\qsim}{\mathrel{\overset{\raisebox{-0.02em}{${q}$}}{\sim}}}

\newtheorem{proposition}{Proposition}[section]
\newtheorem{theorem}[proposition]{Theorem}
\newtheorem{corollary}[proposition]{Corollary}
\newtheorem{lemma}[proposition]{Lemma}
\newtheorem{definition}[proposition]{Definition}

\newtheorem{remark}[proposition]{Remark}
\newtheorem{example}[proposition]{Example}

\newenvironment{tightitemize}{%
    \list{{\textup{$\bullet$}}}{\settowidth\labelwidth{{\textup{\qquad}}}
    \leftmargin\labelwidth \advance\leftmargin\labelsep
    \parsep 0pt plus 1pt minus 1pt \topsep 3pt \itemsep 3pt
    }}{\endlist}

\newenvironment{tightlist}[1]{%
    \list{{\textup{(\roman{enumi})}}}{\settowidth\labelwidth{{\textup{(#1)}}}
    \leftmargin -6pt \advance\leftmargin\labelsep \itemindent \parindent
    \parsep 0pt plus 1pt minus 1pt \topsep 0pt \itemsep 0pt
    \usecounter{enumi}}}{\endlist}

\newenvironment{tightenumerate}[1]{%
    \list{{\bfseries\textup{\arabic{enumi}.}}}{\settowidth\labelwidth{{\textup{(#1)}}}
    \leftmargin -6pt \advance\leftmargin\labelsep \itemindent \parindent
    \parsep 0pt plus 1pt minus 1pt \topsep 0pt \itemsep 0pt
    \usecounter{enumi}}}{\endlist}

\title{Process-Based Risk Measures and Risk-Averse Control of Discrete-Time Systems}

\author{Jingnan Fan\footnote{
Rutgers University, RUTCOR, Piscataway, NJ 08854, USA, Email: {jf546@rutgers.edu}}
 \and Andrzej Ruszczy\'nski\footnote{
Rutgers University, Department of Management Science and Information Systems, Piscataway, NJ 08854, USA, Email: {rusz@rutgers.edu}}}

\date{\small November 5, 2014; revised February 15,2016; revised November 2016}

\begin{document}

\maketitle

\begin{abstract}
For controlled discrete-time stochastic processes we introduce a new class of dynamic  risk measures, which we call process-based. Their main features
are that they measure risk of processes that are functions of the history of a base process. We introduce a new concept of conditional stochastic
time consistency and we derive the structure of process-based risk measures enjoying this property. We show that they can be equivalently
represented by a collection of static law-invariant risk measures on the space of functions of the state of the base process. We apply this
result to controlled Markov processes and we derive dynamic programming equations.
\end{abstract}


\section{Introduction}
The objective of this paper is to provide theoretical foundations of the theory of dynamic risk measures for controlled discrete-time stochastic processes, in particular, Markov processes.

 The theory of dynamic risk measures in discrete time has been intensively developed in the last 10 years (see
 \cite{Scandolo:2003,Riedel:2004,roorda2005coherent,follmer2006convex,CDK:2006,RuSh:2006b,ADEHK:2007,PflRom:07,KloSch:2008,jobert2008valuations,cheridito2011composition}
 and the references therein). The basic setting is the following: we have a probability space $(\varOmega,\Fc,P)$, a filtration
 $\{\Fc_t\}_{t=1,\dots,T}$ with a trivial $\Fc_1$, and we define appropriate spaces $\Zc_t$ of $\Fc_t$-measurable random variables, $t=1,\dots,T$.
 For each $t=1,\dots,T$, a mapping $\rho_{t,T}:\Zc_T\to \Zc_t$ is called a \emph{conditional risk measure}. The central role in the theory
 is played by the concept of \emph{time consistency},  which regulates relations between the mappings $\rho_{t,T}$ and $\rho_{s,T}$ for different $s$ and $t$.
 One definition employed in the literature is the following: \emph{for all $Z,W\in \Zc_T$, if $\rho_{t,T}(Z) \le \rho_{t,T}(W)$ then $\rho_{s,T}(Z) \le \rho_{s,T}(W)$
 for all $s<t$}. This can be used to derive recursive relations $\rho_{t,T}(Z) = \rho_t\big(\rho_{t+1,T}(Z)\big)$, with simpler
\emph{ one-step conditional risk mappings} $\rho_t:\Zc_{t+1}\to\Zc_t$, $t=1,\dots,T-1$. Much effort has been devoted to derive dual representations
of the conditional risk mapping and to study their evolution in various settings.

When applied to  processes described by controlled kernels, in particular, to Markov processes, the theory of dynamic measures of risk encounters difficulties. The spaces $\Zc_t$ are different for different $t$, and thus each one-step mapping $\rho_t$ has different domain and range spaces. With $\Zc_t$ containing all $\Fc_t$ measurable random variables, arbitrary dependence
of $\rho_t$ on the past is allowed.  Moreover, no satisfactory theory of
law invariant dynamic risk measures exists, which would be suitable for Markov control problems (the restrictive
definitions of law invariance employed in \cite{kupper2009representation} and \cite{shapiro2012time} lead to conclusions of limited practical
usefulness, while the translation of the approach of \cite{weber2006distribution} to the Markov case appears to be difficult).
These difficulties are compounded in the case of controlled processes, when a control policy changes the probability measure
on the space of paths of the process. Risk measurement of the entire family of processes defined by control policies is needed.

Motivated by these issues, in
\cite{Ruszczynski2010Markov},  we introduced a specific class of dynamic risk measures, which is well-suited for Markov problems.
We postulated that the one-step conditional risk mappings $\rho_t$ have a special form, which allows for their representation in terms of static risk measures on a space of functions defined on the state space of the Markov process. This restriction allowed for the development of dynamic programming equations and corresponding solution methods, which generalize the well-known results for expected value problems. Our ideas were successfully extended in \cite{ccavus2014transient,ccavus2014computational,lin2013dynamic,shen2013risk}. However, our construction of the Markov risk measures appeared somewhat arbitrary.

In this paper, we introduce and analyze a general class of risk measures, which we call \emph{process-based}. We consider a controlled
process $\{X_t\}_{t=1,\dots,T}$ taking values in a Polish space $\Xc$ (the \emph{state space}),
whose conditional distributions are described by controlled history-dependent
transition kernels
\[
Q_t: {\Xc}^t \times \Uc \rightarrow \Pc(\Xc),\quad t=1,\dots,T-1,
\]
 where
 \[
 \Xc^t = \underbrace{\Xc \times \dots \times \Xc}_{\text{$t$ times}},
  \]
  and $\Uc$ is a certain \emph{control space}. Any history-dependent (measurable) control $u_t=\pi_t(x_1,\dots,x_t)$ is allowed.
 In this setting, we are only interested in measuring risk of  stochastic processes
of the form $Z_t= c_t(X_t,u_t)$, $t=1,\dots,T$, where $c_t:\Xc\times\Uc \to\R$ can be any bounded measurable function.
This restriction of the
class of stochastic processes for which risk needs to be measured is one of the two cornerstones of our approach. The other cornerstone
is our new concept of \emph{stochastic conditional time-consistency}. It is more restrictive than the usual time consistency, because it
involves conditional distributions and uses stochastic dominance rather than the pointwise order. These two foundations allow for the
development of a theory of dynamic risk measures which can be fully described by a sequence of
static law invariant risk measures on a space~$\Vc$
of measurable functions on the state space $\Xc$. In the special case of controlled Markov processes, we derive the structure postulated
in \cite{Ruszczynski2010Markov}, thus providing its solid theoretical foundations. We also derive dynamic programming equations in a much more
general setting than that of \cite{Ruszczynski2010Markov}.

In the extant literature, three basic approaches to introduce risk aversion in  Markov decision processes have been employed:
utility functions (see, e.g., \cite{Jaquette1973,Jaquette1975,Denardo1979,bauerle2013more,jaskiewicz2013persistently}), mean--variance models (see. e.g.,
\cite{White1988,Filar1989,Mannor2013,arlotto2014markov}), and entropic (exponential) models (see, e.g.,
\cite{Howard1971,Marcus1997,bielecki1999risk,Coraluppi1999,DiMasi1999,Levitt2001,bauerle2013more}). Our approach generalizes the utility and exponential models;
the mean--variance models do not satisfy, in general, the monotonicity and time-consistency conditions, except the version of \cite{ChenLiZhao2014}.

The paper is organized as follows. In sections \ref{s:preliminaries}--\ref{s:consistency},
 we formalize the basic model and review the concepts of risk measures and their time consistency.
The first set of original results are presented in section~\ref{s:sctc-trm};
we introduce a new concept of stochastic conditional time consistency and we characterize
the structure of dynamic risk measures enjoying this property (Theorem \ref{prop-li-risk-trans-map}). In section \ref{s:controlled-processes},
we extend these ideas to the case of
controlled processes, and we prove Theorem \ref{th-contr-li-risk-trans-map} on the structure of measures of risk in this case. These results are further specialized to controlled Markov processes in section \ref{s:Markov}. We introduce the concept of a Markov risk measure and we derive its structure
(Theorem \ref{th-markov2}). In section \ref{ss:markov-DP}, we prove an analog of dynamic programming equations in this case.

\section{Risk Measures Based on Observable Processes}
\label{s:processes}

In this section we introduce fundamental concepts and properties of dynamic risk measures for uncontrolled stochastic processes in discrete time. In subsection \ref{s:preliminaries} we set up our probabilistic framework, and in subsections \ref{s:dynamic} and \ref{s:consistency} we revisit some  important concepts existing in the literature. In subsection \ref{s:sctc-trm}, we introduce the new notion of stochastic conditional time consistency, which is a stronger requirement on the dynamic risk measure than the standard time consistency, and which is particularly useful for controlled stochastic processes. Based on this concept, we derive the structure of dynamic risk measures involving transition risk mappings: a family of static risk measures on the space of functions of a state.

\subsection{Preliminaries}
\label{s:preliminaries}

In all subsequent considerations, we work with a Polish space subset $\Xc$ and the canonical measurable space $\left(\Xc^T, \Bc(\Xc)^T\right)$ where $T$ is a natural number and $\Bc(\Xc)^T$ is the product $\sigma$-algebra of Borel sets. We use $\{X_t\}_{t = 1,\dots, T}$ to denote the discrete-time process of canonical projections. We also define $\Hc_t = \Xc^t$ to be the space of possible histories up to time $t$, and we use $h_t$ for a generic element of $\Hc_t$: a specific history up to time $t$. The random vector $(X_1, \cdots, X_t)$ will be denoted by $H_t$.

We assume that for all $t=1,\dots, T-1$, the transition kernels, which describe the conditional distribution of $X_{t+1}$, given $X_1, \cdots, X_t$, are measurable functions
\begin{equation}
	\label{kernel}
	Q_t: {\Xc}^t \rightarrow \Pc(\Xc),\quad t=1,\dots,T-1,
\end{equation}
where $\Pc(\Xc)$ is the set of probability measures on $(\Xc,\Bc(\Xc))$.  These kernels, along with the initial distribution of $X_1$, define a unique probability measure $P$ on the product space $\Xc^T$ with the product $\sigma$-algebra.

For a stochastic system described above, we consider a sequence of random variables $\{Z_t\}_{t = 1 ,\dots, T}$ taking values in ${\R}$; we assume that lower values of $Z_t$ are preferred (e.g., $Z_t$ represents a ``cost'' at time $t$). We require $\{Z_t\}_{t = 1 ,\dots, T}$ to be bounded and adapted to $\{\Fc_t\}_{t = 1 ,\dots, T}$ - the natural filtration generated by the process $X$.  In order to facilitate our discussion, we introduce the following spaces:
\begin{equation}
	\label{Zt}
	\Zc_t = \left\{ \, Z: \Xc^T \rightarrow {\R} \, \big| \, Z \textup{ is } \Fc_t  \textup{-measurable and bounded} \right\}, \quad t=1,\dots,T.
\end{equation}
It is then equivalent to say that $Z_t \in {\Zc}_t$. We also introduce the spaces
\begin{equation*}
	{\Zc}_{t,T} = {\Zc}_t \times \dots \times {\Zc}_T,\quad t=1,\dots,T.
\end{equation*}

Since $Z_t$ is $\Fc_t$-measurable, a measurable function $\phi_t : {\Xc}^t \rightarrow {\R}$ exists such that
$Z_t=\phi_t(X_1,\dots,X_t)$. With a slight abuse of notation, we still use $Z_t$ to denote this function.

\subsection{Dynamic risk measures}
\label{s:dynamic}
In this subsection, we quickly review some definitions and concepts related to risk measures.  All relations (e.g., equality, inequality) between random variables are understood in the ``everywhere'' sense.

\begin{definition} \label{def_CRM}
	A mapping $\rho_{t,T}: {\Zc}_{t,T} \to {\Zc}_t$, where $1 \le t \le T$, is called a \emph{conditional risk measure}, if it has the \emph{monotonicity property}: for all $(Z_t,\dots,Z_T)$ and $(W_t,\dots,W_T)$ in  ${\Zc}_{t,T}$, if $Z_s \le W_s$, for all $s=t, \dots, T$, then $
\rho_{t,T}(Z_t,\dots,Z_T) \le \rho_{t,T}(W_t,\dots,W_T)$.
\end{definition}

\begin{definition} \label{basicProperties} A conditional risk measure $\rho_{t,T}: {\Zc}_{t,T} \to {\Zc}_t$
\begin{tightlist}{ii}
\item is \emph{normalized} if $\rho_{t,T}(0,\dots,0)=0$;
\item is \emph{translation-invariant} if for all $(Z_t,\dots,Z_T) \in {\Zc}_{t,T}$,
		$$\rho_{t,T}(Z_t,\dots,Z_T)=Z_t+\rho_{t,T}(0,Z_{t+1},\dots,Z_T).$$
\end{tightlist}
\end{definition}
Throughout the paper, we assume all conditional risk measures to be at least normalized. Translation-invariance is a fundamental property, which will also be frequently used; under normalization,  it implies that $\rho_{t,T}(Z_t,0, \cdots, 0)=Z_t$.
\begin{definition}\label{LP}
	A conditional risk measure $\rho_{t,T}$ has the \emph{local property} if
	$$\1_A \rho_{t,T}(Z_t,\dots,Z_T)=\rho_{t,T}(\1_A Z_t,\dots,\1_A Z_T),$$
for all $(Z_t,\dots,Z_T) \in {\Zc}_{t,T}$ and for all events $A \in {\Fc}_t$.
\end{definition}
The local property means that the conditional risk measure at time $t$ restricted to any ${\Fc}_t$-event $A$ is not influenced by the values that $Z_t,\dots,Z_T$ take on $A^c$.
\begin{definition} \label{def_DRM}
	A \emph{dynamic risk measure} $\rho=\big\{\rho_{t,T}\big\}_{t = 1, \dots, T} $ is a sequence of conditional risk measures $\rho_{t,T}: {\Zc}_{t,T} \to {\Zc}_t$. We say that $\rho$ is normalized, translation-invariant, or has the local property, if all $\rho_{t,T}$, $t=1,\dots,T$, satisfy the respective conditions of Definitions \ref{basicProperties} or \ref{LP}.
\end{definition}

\subsection{Time Consistency}
\label{s:consistency}
The notion of time consistency can be formulated in different ways, with weaker or stronger assumptions; but the key idea is that if one sequence of costs, compared to another sequence, has the same current cost and lower risk in the future, then it should have lower current risk. In this and the next subsection, we discuss two formulations of time consistency: the (now) standard one, and our new proposal specially suited for process-based measures. We  also show how the tower property (the recursive relation between $\rho_{t,T}$ and $\rho_{t+1,T}$ that the time consistency implies) improves with the more refined time consistency concept. The following
definition of time consistency was employed in \cite{Ruszczynski2010Markov}.

\begin{definition}\label{tc}
	A dynamic risk measure $\big\{\rho_{t,T}\big\}_{t=1,\dots,T}$ is \emph{time-consistent} if for any $1 \le t < T$ and for all $(Z_t,\dots,Z_T), (W_t,\dots,W_T) \in {\Zc}_t$, the conditions
	\[
	\left\{ \begin{aligned} &Z_t=W_t,\\ &\rho_{t+1,T}(Z_{t+1},\dots,Z_T) \le \rho_{t+1,T}(W_{t+1},\dots,W_T),\\ \end{aligned} \right.
	\]
	imply that $\rho_{t,T}(Z_t,\dots,Z_T) \le \rho_{t,T}(W_t,\dots,W_T)$.
\end{definition}

%

It turns out that a translation-invariant and time-consistent dynamic risk measure can be decomposed into and then reconstructed from so-called \emph{one-step conditional risk mappings}.

\begin{theorem}[\cite{Ruszczynski2010Markov}]
	\label{t:one-step}
	A dynamic risk measure $\big\{\rho_{t,T}\big\}_{t=1,\dots,T}$ is translation-invariant and time-consistent if and only if there exist mappings $\rho_t: {\Zc}_{t+1} \to {\Zc}_t$, $t=1,\dots, T-1$,  satisfying the monotonicity and normalization properties, called \emph{one-step conditional risk mappings}, such that for all $t=1,\dots,T-1$,
	\begin{equation} \label{one-step}
		\rho_{t,T}(Z_t,\dots,Z_T)=Z_t+\rho_t\big(\rho_{t+1,T}(Z_{t+1},\dots,Z_T)\big).
	\end{equation}
\end{theorem}
This relation ir related to the \emph{Koopmans equation} \cite{koopmans1960stationary} for utility functions. The operators
$A(Z_{t},Z_{t+1}) = Z_t+ \rho_t(Z_{t+1})$ generalize the concept of \emph{aggregator}
to  measures of risk.

In general, time consistency does not imply the local property, unless additional conditions are satisfied.

Conceptually, the one-step conditional risk mappings play a similar role to one-step conditional expectations, and will be very useful when an analog of the tower property is involved. At this stage, without further refinement of assumptions, it remains a fairly abstract and general object that is hard to  characterize. In \cite{jaskiewicz2013persistently}, for the case of the expected utility model,
$\rho_t$ was a conditional expectation of a pointwise monotonic transformation of its argument.
In \cite{Ruszczynski2010Markov}, a more general, but seemingly special form of this one-step conditional risk mappings was imposed, which was well suited for Markovian applications, but it was unclear whether  other forms of such mappings exist. In order to gain deeper understanding of these concepts,  we introduce a stronger notion of time consistency, and we argue that any  one-step conditional risk mapping is of the form postulated in \cite{Ruszczynski2010Markov}. To this end, we use the particular structure of the space  $\left(\Xc^T, \Bc(\Xc)^T\right)$ and the way a probability measure is defined on this space.

\subsection{Stochastic Conditional Time-Consistency and Transition Risk Mappings}
\label{s:sctc-trm}
We now refine the concept of time-consistency for process-based risk measures.
\begin{definition}
	\label{sc-tc}
	A dynamic risk measure $\big\{\rho_{t,T}\big\}_{t=1,\dots,T}$ is  \emph{stochastically conditionally time-consistent with respect to $\{Q_t\}_{t=1,\dots,T-1}$} if for any $1 \leq t \leq T-1$, for any $h_t \in \Xc^t$, and for all $(Z_t,\dots,Z_T),\, (W_t,\dots,W_T) \in {\Zc}_{t,T}$, the conditions
	\begin{equation}
		\label{sc-tec-cond}
		\left\{
		\begin{aligned}
			&Z_t(h_t)=W_t(h_t),\\
			&\big( \rho_{t+1,T}(Z_{t+1},\dots,Z_T)\mid H_t=h_t \big) \preceq_{\text{\rm st}}
			\big( \rho_{t+1,T}(W_{t+1},\dots,W_T)\mid H_t=h_t \big),\\
		\end{aligned}
		\right.
	\end{equation}
	imply
	\begin{equation}
		\label{tc-impl}
		\rho_{t,T}(Z_t,\dots,Z_T)(h_t) \le \rho_{t,T}(W_t,\dots,W_T)(h_t),
	\end{equation}
	where the relation $\preceq_{\text{\rm st}}$ is the conditional stochastic order understood as follows:
	\begin{multline*}
		\lefteqn{Q_t(h_t)\Big( \big\{ \, x \, | \, \rho_{t+1,T}(Z_{t+1},\dots,Z_T)(h_t,x)> \eta \, \big\} \Big)}\\
		\le
		Q_t(h_t)\Big( \big\{ \, x \, | \, \rho_{t+1,T}(W_{t+1},\dots,W_T)(h_t,x)> \eta \, \big\} \Big),\quad \forall\,\eta\in \R.
	\end{multline*}
\end{definition}
	When the choice of the underlying transition kernels is clear from the context, we will simply say that the dynamic risk measure is stochastically conditionally time-consistent.

\begin{proposition}
	If a dynamic risk measure $\big\{\rho_{t,T}\big\}_{t=1,\dots,T}$ is stochastically conditionally time-consistent and has the translation property, then it is  time-consistent and has the local property.
\end{proposition}

\begin{proof} If  $\big\{\rho_{t,T}\big\}_{t=1,\dots,T}$ is stochastically conditionally time-consistent,
	then it satisfies Definition \ref{tc} and is time-consistent.
	
	Let us prove by induction on $t$ from $T$ down to $1$ that $\rho_{t,T}$ have the local property.
	Clearly,  $\rho_{T,T}$ does: if $A \in {\Fc}_T$, then Definition \ref{basicProperties} yields
	$\1_A\rho_{T,T}(Z_T)=\1_AZ_T=\rho_{T,T}(\1_AZ_T)$.
	
	Suppose  $\rho_{t+1,T}$ satisfies the local property for some $1\le t <T$, and consider
	any $A \in {\Fc}_t$, any $h_t \in \Xc^t$, and any $(Z_t,\dots,Z_T) \in {\Zc}_{t,T}$.
	Two cases may occur.
	\begin{tightitemize}
		\item If $\1_A(h_t)=0$, then $[\1_AZ_t](h_t)=0$. The local property for $t+1$ yields:
		\[
		\big[\rho_{t+1,T}(\1_AZ_{t+1},\dots,\1_AZ_T)\big](h_t,\cdot) = \big[\1_A\rho_{t+1,T}(Z_{t+1},\dots,Z_T)\big](h_t,\cdot) = 0.
		\]
		By stochastic conditional time consistency,
		\[
\rho_{t,T}(\1_AZ_t,\1_A Z_{t+1},\dots,\1_AZ_T) (h_t) = \rho_{t,T}(0, \dots, 0)(h_t) = 0.
\]
		\item If $\1_A(h_t)=1$, then $[\1_AZ_t](h_t)=Z_t(h_t)$. The local property for $t+1$ implies that
		\begin{align*}
			\big[\rho_{t+1,T}(\1_AZ_{t+1},\dots,\1_AZ_T)\big](h_t,\cdot) &= \big[\1_A\rho_{t+1,T}(Z_{t+1},\dots,Z_T)\big](h_t,\cdot)\\
			&= \rho_{t+1,T}(Z_{t+1},\dots,Z_T)(h_t,\cdot).
		\end{align*}
		By stochastic conditional time consistency,
		\[
\rho_{t,T}(\1_AZ_t,\dots,\1_AZ_T) (h_t) = \break \rho_{t,T}(Z_t,\dots,Z_T)(h_t).
\]
	\end{tightitemize}
	In both cases, $\rho_{t,T}(\1_AZ_t,\dots,\1_AZ_T) (h_t) = \big[\1_A\rho_{t,T}(Z_t,\dots,Z_T)\big](h_t)$.
  \end{proof}


Generally, the definition of dynamic risk measures in Section \ref{s:dynamic} and the definition of time consistency property in Section \ref{s:consistency} are valid with any filtration on an underlying probability space $(\Omega,\Fc,P)$, instead of process-generated filtration. However, from now on, we only consider dynamic risk measures defined with a filtration generated by a \emph{specific} process, because we are interested in those risk measures which can be evaluated on each specific history path. That is why we call these risk measures ``process-based.''

The following proposition shows that the stochastic conditional time consistency implies that
the one-step risk mappings $\rho_t$ can be equivalently represented by static law-invariant risk measures on~$\Vc$,
the set of all bounded measurable functions on $\Xc$. We first slightly refine the standard concept of law invariance.

\begin{definition}
\label{d:law-invariance}
A risk measure $r:\Vc\to\R$ is law invariant with respect to the probability measure $q$ on $(\Xc,\Bc(\Xc))$, if for all $V,W\in \Vc$
\[
V \qsim W \Rightarrow r(V) = r(W),
\]
where $V\qsim W$ means that $q\{ V \le \eta\} = q\{ W \le \eta\}$ for all $\eta\in \R$.
\end{definition}

We can now state the main result of this section.
\begin{theorem}
	\label{prop-li-risk-trans-map}
	A process-based dynamic risk measure $\big\{\rho_{t,T}\big\}_{t=1,\dots,T}$ is transla\-tion-invariant and stochastically conditionally time-consistent if and only if functionals
	$\sigma_t: \graph(Q_t) \times \Vc \to {\R}$, $t=1,\dots, T-1$,
	exist, such that
	\begin{tightlist}{ii}
		\item for all $t=1,\dots,T-1$ and all $h_t \in \Xc^t$, the functional
		$\sigma_t(h_t,Q_t(h_t),\cdot)$ is a normalized, monotonic, and law-invariant risk measure on $\Vc$ with respect to the distribution
		$Q_t(h_t)$;
		\item for all $t=1,\dots,T-1$, for all $(Z_t,\dots,Z_T)\in {\Zc}_{t,T}$, and for all $h_t \in \Xc^t$,
		\begin{equation}
\label{li-risk-trans-map}
\rho_{t,T}(Z_t,\dots,Z_T)(h_t)=Z_t(h_t)+\sigma_t\big(h_t,Q_t(h_t),\rho_{t+1,T}(Z_{t+1},\dots,Z_T)(h_t,\cdot)\big).
		\end{equation}
	\end{tightlist}
	Moreover, for all $t=1,\dots,T-1$, $\sigma_t$ is uniquely determined by $\rho_{t,T}$ as follows:  for every $h_t \in \Xc^t$ and every $v\in \Vc$,
	\begin{equation}
		\label{sigma2}
		\sigma_t(h_t,Q_t(h_t),v)=\rho_{t,T}(0,V,0,\dots,0)(h_t),
	\end{equation}
	where $V\in \Zc_{t+1}$ satisfies the equation $V(h_t,\cdot) = v(\cdot)$, and can be arbitrary elsewhere.
\end{theorem}

\begin{proof}
	Assume $\big\{\rho_{t,T}\big\}_{t=1,\dots,T}$ is translation-invariant and stochastically conditionally time-consistent.
We shall prove the existence of $\sigma_t$ satisfying \eqref{li-risk-trans-map}--\eqref{sigma2}.

 Formula \eqref{sigma2} defines a normalized and monotonic risk measure  on the space $\Vc$. Define, for a fixed $h_t\in \Xc^t$,
	\begin{align*}
		v(x) &=
		\rho_{t+1,T}(Z_{t+1},\dots,Z_T)(h_t,x), \quad \forall\,x\in \Xc,\\
		V(h_{t+1}) &= \begin{cases}
			v(x), & \textup{if } h_{t+1} = (h_t,x),\\
			0, & \textup{otherwise}.
		\end{cases}
	\end{align*}
By translation invariance and normalization,
	\[
	\rho_{t+1,T}(V,0,\dots,0)(h_t,\cdot) = V(h_t,\cdot) =   \rho_{t+1,T}(Z_{t+1},\dots,Z_T)(h_t,\cdot).
	\]
	Thus, by the translation property and stochastic conditional time consistency,
	\begin{align*}
		\rho_{t,T}(Z_t,\dots,Z_T)(h_t) &= Z_t(h_t)+\rho_{t,T}(0,Z_{t+1},\dots,Z_T)(h_t)\\
		&= Z_t(h_t)+\rho_{t,T}(0,V,0,\dots,0)(h_t)\\
		&= Z_t(h_t)+\sigma_t(h_t,Q_t(h_t),v).
	\end{align*}
	This chain of relations proves also the uniqueness of $\sigma_t$. We need only verify the postulated law invariance
	of $\sigma_t(h_t,Q_t(h_t),\cdot)$. If $V,V'\in \Zc_{t+1}$ have the same conditional distribution, given $h_t$, then Definition \ref{sc-tc} implies
	that $\rho_{t,T}(0,V,0,\dots,0)(h_t) = \rho_{t,T}(0,V',0,\dots,0)(h_t)$, and law invariance follows from \eqref{sigma2}.
	
	On the other hand, if such transition risk mappings exist, then $\big\{\rho_{t,T}\big\}_{t=1,\dots,T}$ is stochastically
conditionally time-consistent by the monotonicity and law invariance  of $\sigma(h_t,\cdot)$. We can now use \eqref{li-risk-trans-map}
	to obtain for any $t=1,\dots,T-1$, and for all $h_t\in \Xc^t$ the following identity:
	\begin{align*}
		\rho_{t,T}(0,Z_{t+1},\dots,Z_T)(h_t) &= \sigma_t\big(h_t,Q_t(h_t),\rho_{t+1,T}(Z_{t+1},\dots,Z_T)(h_t,\cdot)\big)\\
		&= \rho_{t,T}(Z_t,\dots,Z_T)(h_t)-Z_t(h_t),
	\end{align*}
	which is translation invariance of $\rho_{t,T}$.
  \end{proof}

\begin{remark}
	With a slight abuse of notation, we included the distribution $Q_t(h_t)$ as an argument of the transition risk mapping in view of the application to controlled processes.
\end{remark}

\begin{example}
\label{e:entropic}
In the theory of risk-sensitive Markov decision processes, the following family of \emph{entropic risk measures} is employed
(see \cite{Howard1971,Marcus1997,dai1998explicit,runggaldier1998concepts,bielecki1999risk,Coraluppi1999,DiMasi1999,Levitt2001,bauerle2013more}):
\[
\rho_{t,T}(Z_t,\dots,Z_T) = \frac{1}{\gamma} \ln \bigg(\Eb\Big[\exp\Big(\gamma \textstyle{\sum_{s=t}^T} Z_s\Big)\;\big|\; \Fc_t\Big]\bigg),\quad t=1,\dots,T,\quad \gamma>0.
\]
It is stochastically conditionally time-consistent, and corresponds to the transition risk mapping
\begin{equation}
\label{trm-entropic}
\sigma_t(h_t,q,v)  =  \frac{1}{\gamma}\ln\big( \Eb_q[e^{\gamma v}]\big)
= \frac{1}{\gamma}\ln\bigg( \int_\Xc e^{\gamma v(x)}\;q(dx)\bigg),\quad \gamma>0.
\end{equation}
In the
		construction of a dynamic risk measure, we use $q=Q_t(h_t)$.
We could also make $\gamma$ in \eqref{trm-entropic} dependent on the time $t$, the current state $x_t$, or even the entire
history $h_t$, and still obtain a stochastically conditionally time-consistent dynamic risk measure.
If $\gamma$ depends on $t$ and $x_t$ only, the mapping
\eqref{trm-entropic} corresponds to a Markov risk measure discussed in sections \ref{ss:markov-risk} and \ref{ss:markov-DP}.
\end{example}

	\begin{example}
		\label{e:semi}
		The following transition risk mapping satisfies the condition of Theorem \ref{prop-li-risk-trans-map}
and corresponds to a stochastically conditionally time-consistent dynamic risk measure:
		\begin{equation}
			\label{semi}
			\sigma_t(h_t,q,v) = \int_{\Xc} v(s)\;q(ds) +
\varkappa_t(h_t) \left(\int_{\Xc} \Big[\Big(  v(s)- \int_{\Xc} v(s')\;q(ds')\Big)_+\Big]^p\;q(ds)\right)^{1/p},
		\end{equation}
		where $\varkappa_t:\Xc^t\to[0,1]$ is a measurable function, and $p\in [1,+\infty)$. It is an analogue of the static mean--semideviation measure of risk,
		whose consistency with stochastic dominance is well--known \cite{ogryczak1999stochastic,ogryczak2001consistency}. In the
		construction of a dynamic risk measure, we use $q=Q_t(h_t)$. If $\varkappa_t$ depends on $x_t$ only, the mapping
\eqref{semi} corresponds to a Markov risk measure (see sections \ref{ss:markov-risk} and \ref{ss:markov-DP}).
	\end{example}

\begin{example}
\label{e:avar-sigma}
The following transition risk mapping is derived from
the Average Value at Risk \cite{rockuryas}:
		\begin{equation}
			\label{avar}
			\sigma_t(h_t,q,v) = \min_{\eta\in \R}
\left\{\eta + \frac{1}{\alpha_t(h_t)}\int_{\Xc} \big(  v(s)- \eta \big)_+\;q(ds)\right\},
		\end{equation}
		where $\alpha_t(h_t)$ is a measurable function with values in
$[\alpha_{\min},\alpha_{\max}]\subset (0,1)$. The mapping
\eqref{avar} satisfies the condition of Theorem \ref{prop-li-risk-trans-map};
its consistency with stochastic dominance is well--known \cite{ogryczak2002dual}.
\end{example}

	\begin{example}
		\label{e:non-consistent}
		Our use of the stochastic dominance relation in the definition of stochastic conditional time consistency rules out
		some candidates for transition risk mappings. Suppose
		$\sigma_t(h_t,q,v) = v(x_1)$,
		where $x_1\in \Xc$ is a selected state. Such a mapping is a coherent measure of risk, as a function of the last argument, and may be law invariant. In particular, it is law invariant with
		$\Xc=\{x_1,x_2\}$, $q(x_1) = 1/3$, $q(x_2)=2/3$, $v(x_1) = 3$, $v(x_2) = 1$, $w(x_1) = 2$, $w(x_2) = 4$.
		For this mapping, we have $v \preceq_{\text{\rm st}} w$ under $q$, but $\sigma_t(h_t,q,v) > \sigma_t(h_t,q,w)$, and thus
		the condition of stochastic conditional time consistency is violated. This is due to the fact that the probability of reaching $x_1$, no matter how small, does not
affect the value of the risk measure.
We consciously exclude such cases, because in controlled systems, to be discussed in the next section, the second argument ($q$) is the only one that depends on our decisions. It should be included in the definition of our preferences, if practically meaningful results are to be obtained.
	\end{example}

\section{Risk Measures for Controlled Stochastic Processes}
\label{s:controlled-processes}

We now extend the setting of Section \ref{s:processes} by allowing the kernels \eqref{kernel} to depend on control variab\-les~$u_t$.

\subsection{The Model}
\label{ss:controlled-stochastic}

We still work with the process $\{X_t\}_{t=1 ,\dots, T}$ on the space $\Xc^T$ and introduce a Borel control space $\Uc$. At each time $t$, we observe the state $x_t$ and then apply a control $u_t \in \Uc $.  We assume that the admissible control sets and the transition kernels (conditional distributions of the next state) depend on all currently-known state  and control values. More precisely, we make the
following assumptions:
\begin{tightenumerate}{1}
	\item For all $t=1,\dots, T$, we require that $u_t\in \Uc_t(x_1,u_1,\dots,x_{t-1},u_{t-1},x_t)$, where  $\Uc_t : \Gc_t \rightrightarrows \Uc$  is a measurable multifunction, and  $\Gc_1, \dots, \Gc_T$ are the sets of histories of all currently-known state and control values before applying each control:
	$$\left\{
	\begin{aligned}
	& \Gc_1=\Xc,\\
	& \Gc_{t+1} = \graph(\Uc_t) \times \Xc \subseteq (\Xc \times \Uc)^t \times \Xc,\quad t=1,\dots,T-1;
	\end{aligned}
	\right.
	$$
	\item For all $t=1,\dots, T$, the control-dependent transition kernels
	\begin{equation}
	\label{kernel2}
	Q_t: \graph(\Uc_t) \rightarrow \Pc(\Xc),\quad t=1,\dots,T-1,
	\end{equation}
	are measurable, and for all $t=1,\dots,T-1$, for all $(x_1,u_1,\dots,x_t,u_t) \in \graph(\Uc_t)$, $Q_t(x_1,u_1,\dots,x_t,u_t)$ describes the conditional distribution of $X_{t+1}$, given currently-known states and controls.
\end{tightenumerate}

For this controlled process, a (deterministic) \emph{history-dependent admissible policy} $\pi=(\pi_1,\dots,\pi_T)$ is a sequence of measurable selectors, called \emph{decision rules}, $\pi_t: \Gc_t \rightarrow \Uc$ such that $\pi_t(g_t) \in \Uc_t(g_t)$ for all    $g_t \in \Gc_t$. We can easily prove by induction on $t$ that for  an admissible policy $\pi$ each $\pi_t$ reduces to a measurable function on $\Xc^t$, as $u_s = \pi_s(h_s)$ for all $s=1 ,\dots, t-1$. We are still using $\pi_s$ to denote the decision rule, although it is a different function, formally; it will not lead to any misunderstanding. The set  of admissible policies is
\begin{equation}
\label{varPi}
\begin{aligned}
\varPi := & \ \big\{ \, \pi = (\pi_1,\dots,\pi_T) \, | \,\\
&\quad \forall t, \; \pi_t(x_1,\dots,x_t) \in \Uc_t(x_1,\pi_1(x_1),\dots,x_{t-1},\pi_{t-1}(x_1,\dots,x_{t-1}), x_t) \,\big\}.
\end{aligned}
\end{equation}

 For any fixed policy $\pi \in \varPi$, the transition kernels can be rewritten as measurable functions from $\Xc^t$ to $\Pc(\Xc)$:
\begin{equation}
	\label{kernel3}
	Q_t^\pi: (x_1 ,\dots, x_t) \mapsto Q_t\big(x_1,\pi_1(x_1) ,\dots, x_t, \pi_t(x_1 ,\dots, x_t)\big), \; t = 1 ,\dots, T-1,
\end{equation}
just like the transition kernels of the uncontrolled case given in \eqref{kernel}, but indexed by $\pi$. Thus, for any policy $\pi \in \varPi$, we can consider $\{X_t\}_{t = 1 ,\dots, T}$ as an ``uncontrolled'' process, on the probability space $\big( \Xc^T, \Bc(\Xc)^T, P^\pi \big)$ with $P^\pi$ defined by $\{Q_t^\pi\}_{t=1 ,\dots, T-1}$. The process $\{X_t\}$ is adapted to the policy-independent filtration $\{\Fc_t\}_{t = 1 ,\dots, T}$. As before and throughout this paper, $h_t \in \Xc^t$  stands for $(x_1 ,\dots, x_t)$.

We still use the same spaces $\Zc_t$, $t= 1 ,\dots, T$, as defined in \eqref{Zt} for the costs incurred at each stage; these spaces also allow us to consider control-dependent costs as collections of policy-indexed costs in $\Zc_{1,T}$.
Thus, we are able to define and analyze (time-consistent) dynamic risk measures $\rho^\pi$ for each fixed $\pi \in \varPi$, as in Section~\ref{s:processes}. Note that $\rho^\pi$ are defined on the same spaces independently of $\pi$,
 because the filtration and the spaces $\Zc_t, \, t= 1 ,\dots, T$, are not dependent on $\pi$; however, we do need to index the measures of risk by the policy $\pi$,
because the transition kernels and, consequently, the probability measure on the space $\Xc^T$, depend on $\pi$.

\subsection{Stochastic Conditional Time Consistency and Transition Risk Mappings}

 We need to compare risk levels among different policies,
 so a meaningful order among the risk measures $\rho^\pi$, with ${\pi \in \varPi}$, is needed. It turns out that our concept of
stochastic conditional time-consistency can be extended to this setting.

\begin{definition}
	\label{tc-controlled}
A family of process-based dynamic risk measures $\big\{\rho^\pi_{t,T}\big\}_{t=1,\dots,T-1}^{\pi \in \varPi}$ is  \emph{stochastically conditionally time-consistent} if for any $\pi,\pi' \in \varPi$, for any $1 \le t < T$, for all $h_t \in \Xc^t$,
all $(Z_t,\dots,Z_T)\in \Zc_{t,T}$
and all $ (W_t,\dots,W_T)\in \Zc_{t,T}$, the conditions
\[
\left\{ \begin{aligned}
&Z_t(h_t)=W_t(h_t),\\
&\big( \rho^\pi_{t+1,T}(Z_{t+1},\dots,Z_T)\mid H^\pi_t=h_t \big) \preceq_{\text{\rm st}} \big( \rho^{\pi'}_{t+1,T}(W_{t+1},\dots,W_T)\mid H_t^{\pi'}=h_t \big),
\end{aligned} \right.
\]
 imply
\[
\rho^\pi_{t,T}(Z_t,\dots,Z_T)(h_t) \le \rho^{\pi'}_{t,T}(W_t,\dots,W_T)(h_t).
\]
\end{definition}

\begin{remark}
    As in Definition \ref{sc-tc}, the conditional stochastic order ``$\preceq_{\text{\rm st}}$''
    is understood as follows: for all $\eta \in \R$ we have
    \[
   \ Q_t^\pi(h_t)\Big( \big\{ \, x \, | \, \rho^\pi_{t+1,T}(Z_{t+1},\dots,Z_T)(h_t,x)> \eta \, \big\} \Big)
    \le
    Q_t^{\pi'}(h_t)\Big( \big\{ \, x \, | \, \rho^{\pi'}_{t+1,T}(W_{t+1},\dots,W_T)(h_t,x)> \eta \, \big\} \Big).
    \]
\end{remark}

This definition helps us build a connection among dynamic risk measures $\rho^\pi$, for ${\pi \in \varPi}$, as we explain it below. Before passing to the details,
we can say in short that the same transition risk mappings as in the uncontrolled case
are the only possible structures of such risk measures.

 If a family of process-based dynamic risk measures $\big\{\rho^\pi_{t,T}\big\}_{t=1,\dots,T-1}^{\pi \in \varPi}$ is stochastically conditionally time-consistent, then for each fixed $\pi\in\varPi$ the process-based dynamic risk measure $\big\{\rho^\pi_{t,T}\big\}_{t=1,\dots,T-1}$ is  stochastically conditionally time-consistent, as defined in Definition \ref{sc-tc}. By virtue of Proposition \ref{prop-li-risk-trans-map}, for each $\pi\in\varPi$, there exist functionals
$\sigma^\pi_t: \text{\rm graph}(Q^\pi_t) \times \Vc \to {\R}$, $t=1\dots T-1$,
such that for all $t=1,\dots,T-1$, all $h_t \in \Xc^t$, the functional $\sigma_t^\pi(h_t,Q^\pi_t(h_t),\,\cdot\,)$ is a law-invariant risk measure on ${\Vc}$ with respect to the distribution $Q^\pi_t(h_t)$ and
\[
\rho^\pi_{t,T}(Z_t,\dots,Z_T)(h_t)=Z_t(h_t)+\sigma^\pi_t\big(h_t,Q^\pi_t(h_t),\rho^\pi_{t+1,T}(Z_{t+1},\dots,Z_T)(h_t,\cdot)\big),
\quad \forall h_t \in \Xc^t.
\]
Consider any $\pi,\pi' \in \varPi$, $h_t\in \Xc^t$, and $(Z_t,\dots,Z_T)\in \Zc_{t,T}$, $(W_t,\dots,W_T) \in \Zc_{t,T}$ such that
\[
\left\{
    \begin{aligned}
    & Z_t(h_t)=W_t(h_t),\\
    & Q^\pi_t(h_t)=Q^{\pi'}_t(h_t),\\
    & \rho^\pi_{t+1,T}(Z_{t+1},\dots,Z_T)(h_t,\cdot) = \rho^{\pi'}_{t+1,T}(W_{t+1},\dots,W_T)(h_t,\cdot).
    \end{aligned}
\right.
\]
Then we have
\[
\big( \rho^\pi_{t+1,T}(Z_{t+1},\dots,Z_T)\mid H^\pi_t=h_t \big) \sim_{\text{\rm st}} \big( \rho^{\pi'}_{t+1,T}(W_{t+1},\dots,W_T)\mid H_t^{\pi'}=h_t \big),
\]
where the relation $\sim_{\text{\rm st}}$ means that both $\preceq_{\text{\rm st}}$ and $\succeq_{\text{\rm st}}$ are true; in other words, equality in law.
Because of the stochastic conditional time-consistency,
\[
\rho^\pi_{t,T}(Z_t,\dots,Z_T)(h_t) = \rho^{\pi'}_{t,T}(W_t,\dots,W_T)(h_t),
\]
whence
\[
\sigma^\pi_t\big(h_t,Q^\pi_t(h_t),\rho^\pi_{t+1,T}(Z_{t+1},\dots,Z_T)(h_t,\cdot)\big)
=\sigma^{\pi'}_t\big(h_t,Q^{\pi'}_t(h_t),\rho^{\pi'}_{t+1,T}(W_{t+1},\dots,W_T)(h_t,\cdot)\big).
\]
All three arguments of $\sigma^\pi_t$ and $\sigma^{\pi'}_t$ are identical. Consequently, $\sigma^\pi$ does not depend on $\pi$ directly, and all dependence on $\pi$ is carried by the controlled kernel $Q_t^\pi$. This is a highly desirable property, when we apply dynamic risk measures to a control problem.
We summarize this important observation in the following theorem, which extends Theorem \ref{prop-li-risk-trans-map} to the case of controlled processes.

\begin{theorem}
\label{th-contr-li-risk-trans-map}
A family of process-based dynamic risk measures $\big\{\rho_{t,T}^\pi\big\}_{t=1,\dots,T}^{\pi \in \varPi}$ is trans\-lation-invariant and stochastically conditionally time-consistent if and only if there exist functionals
\[
\sigma_t: \bigg\{\bigcup_{\pi \in \varPi} \text{\rm graph}(Q_t^\pi)\bigg\} \times \Vc \to {\R}, \quad t=1\dots T-1,
\]
such that:
\begin{tightlist}{ii}
	\item For all $t=1,\dots,T-1$ and all $h_t \in \Xc^t$, $\sigma_t(h_t,\cdot,\cdot)$ is normalized and
	has the following  property of strong monotonicity with respect to stochastic dominance:
	\begin{multline*}
	\forall q^1,q^2 \in \big\{ \, Q_t^\pi(h_t):\pi \in \varPi \, \big\}, \;
	 \forall v^1,v^2 \in \Vc, \;\\
	(v^1;q^1) \preceq_{\text{st}} (v^2;q^2) \implies \sigma_t(h_t,q^1,v^1) \le \sigma_t(h_t,q^2,v^2),
	\end{multline*}
	where $(v;q) = q \circ v^{-1}$ means ``the distribution of $v$ under $q$;''
	\item For all $\pi \in \varPi$, for all $t=1,\dots,T-1$, for all $(Z_t,\dots,Z_T)\in {\Zc}_{t,T}$, and for all $h_t \in \Xc^t$,
	\begin{equation}
	\label{li-risk-trans-map-controlled} \rho^\pi_{t,T}(Z_t,\dots,Z_T)(h_t)=Z_t(h_t)+\sigma_t(h_t,Q_t^\pi(h_t),\rho_{t+1,T}^\pi(Z_{t+1},\dots,Z_T)(h_t,\cdot)).
	\end{equation}
\end{tightlist}
Moreover, for all $t=1,\dots,T-1$, $\sigma_t$ is uniquely determined by $\rho_{t,T}$ as follows:  for every $h_t \in \Xc^t$, for every $q \in \big\{ \, Q_t^\pi(h_t):\pi \in \varPi \, \big\}$, and for every $v\in \Vc$,
\begin{equation}
\label{sigma3}
\sigma_t(h_t,q,v)=\rho_{t,T}^\pi(0,V,0,\dots,0)(h_t),
\end{equation}
where $\pi$ is any admissible policy such that $q=Q_t^\pi(h_t)$, and $V\in \Zc_{t+1}$ satisfies the equation $V(h_t,\cdot) = v(\cdot)$, and can be arbitrary elsewhere.
\end{theorem}

\begin{proof}
	We have shown the existence of $\{\sigma_t\}_{t=1 ,\dots, T}$ satisfying \eqref{li-risk-trans-map-controlled} and \eqref{sigma3} in the discussion preceding the theorem. We can verify the strong law-invariance by \eqref{sigma3} and Definition \ref{tc-controlled}.
  \end{proof}

It follows that the transition risk mappings of Examples \ref{e:entropic},  \ref{e:semi},
 and \ref{e:avar-sigma} are perfectly suitable transition risk mappings for controlled processes as well,
 provided that the corresponding parameters ($\gamma$, $\varkappa$, and $\alpha$)
 depend on $t$ and $x_t$ only.

\section{Application to Controlled Markov Systems}
\label{s:Markov}

Our results can be further specialized to the case when $\{X_t\}$ is a controlled Markov system, in which we assume the following conditions:
\begin{tightlist}{i}
	\item The admissible control sets are measurable multifunctions of the current state, i.e.,
	$\Uc_t: \Xc \rightrightarrows \Uc, \quad t = 1 ,\dots, T;$
	\item The dependence in the transition kernel \eqref{kernel2} on the history is carried only through the last state and control:
	$Q_t: \graph(\Uc_t) \rightarrow \Pc(\Xc)$, $t=1,\dots,T-1$;
	\item The step-wise costs are  dependent only on the current state and control:
	$Z_t=c_t(x_t,u_t)$, $t=1,\dots,T$,
	where $c_t:\graph(\Uc_t) \rightarrow \R, \, t=1,\dots,T$ are measurable bounded functions.
\end{tightlist}

Let $\varPi$ be the set of admissible \emph{history-dependent policies}:
\begin{equation*}
\varPi := \big\{ \, \pi = (\pi_1,\dots,\pi_T) \, | \, \forall t, \; \pi_t(x_1,\dots,x_t) \in \Uc_t(x_t) \,\big\}.
\end{equation*}
To alleviate notation, for all $\pi \in \varPi$ and for all measurable $c=(c_1, \dots, c_T)$, we write
\begin{equation*}
v_t^{c,\pi}(h_t) := \rho^\pi_{t,T}\big(c_t(X_t,\pi_t(H_t)),\dots,c_T(X_T,\pi_T(H_T))\big)(h_t).
\end{equation*}

The following result is a direct consequence of Theorem \ref{th-contr-li-risk-trans-map} in the Markovian case.
\begin{corollary}
	\label{th-markov1}
	For a controlled Markov system, a family of process-based dynamic risk measures $\big\{\rho_{t,T}^\pi\big\}_{t=1,\dots,T}^{\pi \in \varPi}$ is  translation-invariant  and stochastically conditionally time-consistent if and only if  functionals
	\[
	\sigma_t: \left\{ \big(h_t,Q_t(x_t,u)\big) : h_t \in \Xc^t, u \in \Uc_t(x_t) \right\} \times \Vc \to {\R}, \quad t=1\dots T-1,
	\]
	exist, such that
	\begin{tightlist}{ii}
		\item For all $t=1,\dots,T-1$ and all $h_t \in \Xc^t$, $\sigma_t(h_t,\cdot,\cdot)$ is normalized and strongly monotonic with respect to stochastic dominance on $\big\{ \, Q_t(x_t,u):u \in \Uc_t(x_t) \, \big\}$;
		\item For all $\pi \in \varPi$, for all bounded measurable $c$, for all $t=1,\dots,T-1$, and for all $h_t \in \Xc^t$,
		\begin{equation}
		\label{li-risk-trans-map-markov1} v^{c,\pi}_t(h_t)=c_t(x_t,\pi_t(h_t))+\sigma_t\bigg(h_t,Q_t(x_t,\pi_t(h_t)),v^{c,\pi}_{t+1}(h_t,\cdot)\bigg).
		\end{equation}
	\end{tightlist}
\end{corollary}
\begin{proof}
	To verify the ``if and only if'' statement, we can show that \eqref{sigma3} is true if $\sigma_t$ satisfies \eqref{li-risk-trans-map-markov1} for all measurable bounded $c$.
  \end{proof}

\subsection{Markov Risk Measures}
\label{ss:markov-risk}

Consider a Markov policy $\pi$  composed of state-dependent measurable decision rules $\pi_t:\Xc \mapsto \Uc $, $t=1 ,\dots, T$.
Because of the Markov property of the transition kernels, for a Markov policy $\pi$, the future evolution of the process $\{X_\tau\}_{\tau=t ,\dots, T}$ is solely dependent on the current state $x_t$, so is the distribution of the future costs $c_\tau(X_\tau,\pi_\tau(X_\tau))$, $\tau=t,\dots,T$.
Therefore, it is reasonable to assume that the dependence of the conditional risk measure on the history
is also carried  by the current state only.

\begin{definition}
 \label{d:markov-risk}
 A family of process-based dynamic risk measures $\big\{\rho_{t,T}^\pi\big\}_{t=1,\dots,T}^{\pi \in \varPi}$ for a controlled Markov system is \emph{Markov} if for all Markov policies $\pi \in \varPi$, for all measurable $c=(c_1,\dots,c_T)$, and for all $h_t=(x_1,\dots,x_t)$ and $h_t'=(x_1',\dots,x_t')$ in $\Xc^t$ such that $x_t=x_t'$, we have
$v^{c,\pi}_t(h_t)=v^{c,\pi}_t(h_t')$.
\end{definition}

\begin{proposition}
	\label{mless-sigma-mdp}
	Under translation invariance and  stochastic conditional time consistency, $\big\{\rho_{t,T}^\pi\big\}_{t=1,\dots,T}^{\pi \in \varPi}$ is Markov if and only if the dependence of $\sigma_t$ on $h_t$ is carried only by $x_t$, for all $t=1 ,\dots, T-1$.
\end{proposition}

\begin{proof}
Suppose $\big\{\rho_{t,T}^\pi\big\}_{t=1,\dots,T}^{\pi \in \varPi}$ is Markov. For all $t=1 ,\dots, T-1$, for all $h_t,h_t' \in \Xc^t$ such that $x_t=x_t'$, for all $u \in \Uc_t(x_t)$ and for all $v \in \Vc$, there exists a Markov $\pi \in \varPi$ such that $\pi_t(x_t)=u$. By setting $c=(0 ,\dots, 0,c_{t+1}, 0 ,\dots, 0)$ with $c_{t+1}: (x',u') \mapsto v(x')$, the Markov property of $\rho^\pi$ implies that
\[
\sigma_t(h_t,Q_t(x_t,u),v)=v^{c,\pi}_t(h_t)=v^{c,\pi}_t(h_t')=\sigma_t(h_t',Q_t(x_t,u),v).
\]
Therefore, $\sigma_t$ is indeed memoryless, that is, its dependence on $h_t$ is carried by $x_t$ only.
	
	Suppose $\sigma_t, \, t=1 ,\dots, T-1$, are all memoryless. We prove by induction backward in time that for all $t=T ,\dots, 1$, $v^{c,\pi}_t(h_t)=v^{c,\pi}_t(h_t')$ for all Markov $\pi$ and all $h_t,h_t' \in \Xc^t$ such that $x_t=x_t'$. For $t=T$ we have:
$v_T^{c,\pi}(h_T) = c_T(x_T,\pi_T(x_T))=v_T^{c,\pi}(h'_T)$.
We can just write it as $v_T^{c,\pi}(x_T)$. If this relation is true for some $t+1 \le T$,
then for $t$ we obtain
\begin{align*}
v^{c,\pi}_t(h_t) &= c_t(x_t,\pi_t(x_t))+ \sigma_t\big(x_t,Q_t(x_t,\pi_t(x_t)),v^{c,\pi}_{t+1}(h_t,\cdot)\big)\\
&=
c_t(x_t,\pi_t(x_t))+ \sigma_t\big(x_t,Q_t(x_t,\pi_t(x_t)),v^{c,\pi}_{t+1}(\cdot)\big).
\end{align*}
The right hand side is a function of $x_t$, rather than $h_t$, and we can write the value of the risk
measure as
$v^{c,\pi}_t(x_t)$. By induction, the result holds true for all $t$.
  \end{proof}

\begin{theorem}
	\label{th-markov2}
	For a controlled Markov system, a family of process-based dynamic risk measures $\big\{\rho_{t,T}^\pi\big\}_{t=1,\dots,T}^{\pi \in \varPi}$ is translation-invariant, stochastically conditionally time-consistent, and Markov, if and only if there exist functionals
	\[
	\sigma_t: \left\{ \big(x,Q_t(x,u)\big) : x \in \Xc, u \in \Uc_t(x) \right\} \times \Vc \to {\R}, \quad t=1\dots T-1,
	\]
	where $\Vc$ is the set of bounded measurable functions on $\Xc$, such that:
	\begin{tightlist}{ii}
		\item For all $t=1,\dots,T-1$ and all $x \in \Xc$, $\sigma_t(x,\cdot,\cdot)$ is
normalized and strongly monotonic with respect to stochastic dominance on $\big\{ \, Q_t(x,u):u \in \Uc_t(x) \, \big\}$;
		\item For all $\pi \in \varPi$, for all measurable bounded $c$, for all $t=1,\dots,T-1$, and for all $h_t \in \Xc^t$,
		\begin{equation}
		\label{li-risk-trans-map-markov2} v^{c,\pi}_t(h_t)=c_t(x_t,\pi_t(h_t))+\sigma_t\bigg(x_t,Q_t(x_t,\pi_t(h_t)),v^{c,\pi}_{t+1}(h_t,\cdot)\bigg).
		\end{equation}
	\end{tightlist}
\end{theorem}


Theorem \ref{th-markov2} provides us with a simple recursive formula \eqref{li-risk-trans-map-markov2}
for the evaluation of risk of a {Markov} poli\-cy~$\pi$:
\begin{align*}
v_T^{c,\pi}(x)&=c_T(x,\pi_T(x)),\quad x\in \Xc,\\
v_t^{c,\pi}(x)&=c_t(x,\pi_t(x))+\sigma_t\big(x,Q_t(x,\pi_t(x)),v_{t+1}^{c,\pi}\big),\quad x \in \Xc,\quad t=T-1,\dots,1.
\end{align*}
It involves calculation of the values of functions $v_t^{c,\pi}(\cdot)$ on the state space $\Xc$.

\subsection{Dynamic Programming}
\label{ss:markov-DP}
In this section, we fix the cost functions $c_1,\dots,c_T$ and consider a family of dynamic risk measures $\big\{\rho_{t,T}^\pi\big\}_{t=1,\dots,T}^{\pi \in \varPi}$ which is normalized, translation-invariant
(Defini\-tion~\ref{basicProperties}), stochastically conditionally time-consistent
(Definition \ref{tc-controlled}), and Markov (Defini\-tion~\ref{d:markov-risk}). Our objective
  is to analyze the risk minimization problem:
\[
\min_{\pi \in \varPi} v_1^\pi(x_1),\quad x_1\in\Xc.
\]
For this purpose, we introduce the family of \emph{value functions}:
\begin{equation}\label{p:markov}
v_t^*(h_t)=\inf_{\pi \in \varPi_{t,T}(h_t)} v_t^\pi(h_t), \quad t= 1,\dots,T, \quad h_t \in \Xc^t,
\end{equation}
where $\varPi_{t,T}(h_t)$ is the set of feasible deterministic policies $\pi = \{\pi_t,\dots,\pi_T\}$.
As stated in Theorem \ref{th-markov2},  transition risk mappings $\big\{\sigma_t\big\}_{t=1, \dots, T-1}$ exist, such that
\begin{equation}\label{eq:vc-markov}
v_t^\pi(h_t) = c_t(x_t,\pi_t(h_t))+\sigma_t\big(x_t,Q_t(x_t,\pi_t(h_t)),v_{t+1}^\pi(h_t,\cdot)\big), \quad t = 1,\dots,T-1, \ \pi \in \varPi, \ h_t \in \Xc^t.
\end{equation}
Our intention is to prove that the value functions $v_t^*(\cdot)$ are \emph{memoryless}, that is, for all
$h_t=(x_1,\dots,x_t)$ and $h_t'=(x_1',\dots,x_t')$ such that $x_t=x_t'$, we have
$v_t^*(h_t)=v_t^*(h_t')$. In this case, with a slight abuse of notation, we shall simply write $v_t^*(x_t)$.

In order to formulate the main result of this subsection, we equip the space $\Pc(\Xc)$ of probability measures
on $\Xc$ with the topology of weak convergence.

\begin{theorem} \label{th:markov-DP}
Suppose a family of dynamic risk measures $\big\{\rho_{t,T}^\pi\big\}_{t=1,\dots,T}^{\pi \in \varPi}$ is normalized, translation-invariant, stochastically conditionally time-consistent, and Markov.
We assume the following conditions:
\begin{tightlist}{iii}
\item The transition kernels $Q_t(\cdot,\cdot)$, $t=1,\dots,T$, are weakly continuous;
\item For every lower semicontinuous $v\in \Vc$ the transition risk  mappings $\sigma_t(\cdot,\cdot,v)$, $t=1,\dots,T$, are lower semicontinuous;
\item The functions $c_t(\cdot,\cdot)$, $t=1,\dots,T$, are lower semicontinuous;
\item The multifunctions $\Uc_t(\cdot)$, $t=1,\dots,T$, are compact-valued, {and upper semicontinuous}.
\end{tightlist}
Then the functions $v_t^*$, $t=1,\dots,T$, are memoryless,  lower semicontinuous,
and satisfy the following {dynamic programming equations}:
\begin{gather*}
v_T^*(x)=\min_{u \in \Uc_T(x)} c_T(x,u), \quad x \in \Xc,\\
v_t^*(x)=\min_{u \in \Uc_t(x)} \left\{ c_t(x,u) + \sigma_t\big(x,Q_t(x,u),v_{t+1}^* \big) \right\}, \quad x \in \Xc, \quad t = T-1,\dots,1.
\end{gather*}
Moreover, an optimal Markov policy $\hat{\pi}$ exists and satisfies the equations:
\begin{gather*}
\hat{\pi}_T(x) \in \argmin_{u \in \Uc_T(x)} \,c_T(x,u), \quad x \in \Xc,\\
\hat{\pi}_t(x) \in \argmin_{u \in \Uc_t(x)} \left\{ c_t(x,u) + \sigma_t\big(x,Q_t(x,u),v_{t+1}^*\big) \right\}, \quad x \in \Xc, \quad t = T-1,\dots,1.
\end{gather*}
\end{theorem}

\begin{proof}
We prove the memoryless property of $v_t^*(\cdot)$ and construct the optimal Markov policy by induction backwards in time.
For all $h_T \in \Xc^T$ we have
\begin{equation}
\label{infT}
v_T^*(h_T)=\inf_{\pi \in \varPi} c_T(x_T,\pi_T(h_T))=\inf_{u\in\Uc_T(x_T)} c_T(x_T,u).
\end{equation}
Since $c_T(\cdot,\cdot)$ is lower semicontinuous, it is a \emph{normal integrand}, that is, its epigraphical mapping
$x \mapsto \{ (u,\alpha)\in \Uc\times \R : c_T(x,u) \le \alpha \}$
is closed-valued and measurable \cite[Def. 14.1, Ex. 14.31]{rockafellar1998variational}. Due to assumption (iv),
the mapping
\[
\bar{c}_T(x,u)= \begin{cases}
c_T(x,u) & \text{if } u\in \Uc_T(x),\\
+\infty & \text{otherwise},
\end{cases}
\]
is a normal integrand as well.
By virtue of \cite[Thm. 14.37]{rockafellar1998variational}, the infimum in \eqref{infT} is attained and is a measurable function of $x_T$.
Hence, $v_T^*(\cdot)$ is measurable and memoryless.
{By assumptions (iii) and (iv) and Berge theorem, it is also lower semicontinuous
(see, \emph{e.g.}, \cite[Thm. 1.4.16]{aubin2009set}).} Moreover, the optimal solution mapping $\varPsi_T(x)=\big\{u\in \Uc_T(x): c_T(x,u)=v_T^*(x)\big\}$ is measurable and has nonempty and closed values. Therefore, a measurable selector $\hat{\pi}_T$ of  $\varPsi_T$ exists~\cite{kuratowski1965general}, \cite[Thm. 8.1.3]{aubin2009set}.

Suppose $v^*_{t+1}(\cdot)$ is  memoryless  and lower semicontinuous,
and Markov decision rules $\{\hat{\pi}_{t+1},\dots,\hat{\pi}_T\}$ exist such that
\[
v^*_{t+1}(x_{t+1}) = v_{t+1}^{\{\hat{\pi}_{t+1},\dots,\hat{\pi}_T\}}(x_{t+1}), \quad \forall\, h_{t+1} \in \Xc^{t+1}.
\]
Then for any $h_t \in \Xc^t$ we have
\[
v_t^*(h_t)=\inf_{\pi \in \varPi} v_t^\pi(h_t)=\inf_{\pi \in \varPi} \left\{ c_t(x_t,\pi_t(h_t))+\sigma_t\big(x_t,Q_t(x_t,\pi_t(h_t)),v_{t+1}^\pi(h_t,\cdot)\big) \right\}.
\]

On the one hand, since $v_{t+1}^\pi(h_t,\cdot) \ge v_{t+1}^*(\cdot)$ and $\sigma_t$ is non-decreasing with respect to the last argument, we obtain
\begin{equation}
\label{lowervt}
\begin{aligned}
v_t^*(h_t) &\ge \inf_{\pi \in \varPi} \left\{ c_t(x_t,\pi_t(h_t))+\sigma_t\big(x_t,Q_t(x_t,\pi_t(h_t)),v_{t+1}^*\big) \right\} \\
&= \inf_{u \in \Uc_t(x_t)} \left\{ c_t(x_t,u)+\sigma_t\big(x_t,Q_t(x_t,u),v_{t+1}^*\big) \right\}.
\end{aligned}
\end{equation}
By assumptions (i)--(iii), the mapping
$(x,u)\mapsto c_t(x,u)+\sigma_t(x,Q_t(x,u),v_{t+1}^*)$ is lower semicontinuous. Invoking \cite[Thm. 14.37]{rockafellar1998variational} and assumption (iv) again,
exactly as in the case of $t=T$, we conclude that the optimal solution mapping
\begin{align*}
\varPsi_t(x)&=\Big\{u\in \Uc_t(x): c_t(x,u)+\sigma_t\big(x,Q_t(x,u),v_{t+1}^*\big)\\
&=\inf_{u \in \Uc_t(x)} \left\{ c_t(x,u)+\sigma_t\big(x,Q_t(x,u),v_{t+1}^*\big) \right\}\Big\}
\end{align*}
is measurable and has nonempty and closed values; hence, a measurable selector $\hat{\pi}_t$ of  $\varPsi_t$ exists
\cite[Thm. 8.1.3]{aubin2009set}.
Substituting this selector into~\eqref{lowervt}, we obtain
\[
v_t^*(h_t)\ge
c_t(x_t,\hat{\pi}_t(x_t)) + \sigma_t\big(x_t,Q_t(x_t,\hat{\pi}_t(x_t)),v_{t+1}^{\{\hat{\pi}_{t+1},\dots,\hat{\pi}_T\}}\big) = v_t^{\{\hat{\pi}_t,\dots,\hat{\pi}_T\}}(x_t).
\]
In the last equation, we used \eqref{eq:vc-markov} and the fact that the decision rules $\hat{\pi}_t,\dots,\hat{\pi}_T$ are Markov.

On the other hand,
\[
v_t^*(h_t)=\inf_{\pi \in \varPi} v_t^\pi(h_t) \le v_t^{\{\hat{\pi}_t,\dots,\hat{\pi}_T\}}(x_t).
\]
Therefore, $v_t^*(h_t)=v_t^{\{\hat{\pi}_t,\dots,\hat{\pi}_T\}}(x_t)$ is measurable, memoryless, and
\begin{align*}
v_t^*(x_t)&=\min_{u \in \Uc_t(x_t)} \left\{ c_t(x_t,u) + \sigma_t\big(x_t,Q_t(x,u),v_{t+1}^*\big) \right\}\\
 &= c_t(x_t,\hat{\pi}_t(x_t)) + \sigma_t\big(x_t,Q_t(x_t,\hat{\pi}_t(x_t)),v_{t+1}^*\big).
\end{align*}
{ By assumptions (ii), (iii), (iv), and Berge theorem, $v_t^*(\cdot)$ is lower semicontinuous
(see, e.g. \cite[Thm. 1.4.16]{aubin2009set}).}
This completes the induction step.
  \end{proof}
\begin{remark}
\label{r:continuous}
If we replace semicontinuity with continuity in assumptions (ii)--(iv), then the value functions $v_t^*$, $t=1,\dots,T$, will be continuous. The proof is identical.
\end{remark}

 Let us verify the weak lower semicontinuity assumption (ii) of the mean--semide\-via\-tion transition risk mapping of Example \ref{e:semi}.
 To make the mapping Markovian, we assume that the parameter $\varkappa$ depends on $x$ only, that is,
 \begin{equation}
    \label{semi-lsc}
    \sigma(x,q,v) = \int_{\Xc} v(s)\;q(ds) + \varkappa(x)
    \left(\int_{\Xc} \bigg[\Big( v(s)- \int_{\Xc} v(s')\;q(ds')\Big)_+\bigg]^p\;q(ds)\right)^{{1}/{p}}.
    \end{equation}
    As before, $p\in [1,\infty)$. For simplicity, we skip the subscript $t$ of $\sigma$ and $\varkappa$.
    \begin{lemma}
    \label{l:semi-lsc}
    Suppose $\varkappa(\cdot)$ is continuous. Then for every
    lower semicontinuous function $v$, the mapping $(x,q)\mapsto \sigma(x,q,v)$ in \eqref{semi-lsc} is lower semicontinuous.
    \end{lemma}
    \begin{proof}
 Let $q_k\to q$ weakly and $x_k\to x$. For all $s\in \Xc$ we have the inequality
 \begin{align*}
 0 &\le \bigg[v(s)- \int_{\Xc} v(s')\;q(ds')\bigg]_+ \\
 &\le \bigg[v(s)- \int_{\Xc} v(s')\;q_k(ds')\bigg]_+
 + \bigg[\int_{\Xc} v(s')\;q_k(ds')-\int_{\Xc} v(s')\;q(ds')\bigg]_+.
 \end{align*}
 By the triangle inequality for the norm in ${\Lc}_p(\Xc,\Bc(\Xc),q_k)$,
 \begin{multline*}
 \left(\int_{\Xc} \bigg[v(s)- \int_{\Xc} v(s')\;q(ds')\bigg]_+^p \;q_k(ds)\right)^{{1}/{p}} \\
 \le
 \left(\int_{\Xc} \bigg[v(s)- \int_{\Xc} v(s')\;q_k(ds')\bigg]_+^p \;q_k(ds)\right)^{{1}/{p}}
 + \bigg[\int_{\Xc} v(s')\;q_k(ds')-\int_{\Xc} v(s')\;q(ds')\bigg]_+.
 \end{multline*}
 Adding $\int_{\Xc} v(s)\,q(ds)$ to both sides, we obtain
 \begin{multline*}
 \int_{\Xc} v(s)\;q(ds) + \left(\int_{\Xc} \bigg[v(s)- \int_{\Xc} v(s')\;q(ds')\bigg]_+^p \;q_k(ds)\right)^{{1}/{p}} \\
 \le
 \left(\int_{\Xc} \bigg[v(s)- \int_{\Xc} v(s')\;q_k(ds')\bigg]_+^p \;q_k(ds)\right)^{{1}/{p}}
 + \max\bigg[\int_{\Xc} v(s)\;q_k(ds),\int_{\Xc} v(s)\;q(ds)\bigg].
 \end{multline*}
    By the lower semicontinuity of $v$ and weak convergence of $q_k$ to $q$, we have
    \[
    \int_{\Xc} v(s)\;q(ds) \le \liminf_{k\to \infty} \int_{\Xc} v(s)\;q_k(ds),
    \]
    that is, for every $\varepsilon>0$, we can find $k_\varepsilon$ such that for all $k \ge k_\varepsilon$
    \[
    \int_{\Xc} v(s)\;q_k(ds) \ge \int_{\Xc} v(s)\;q(ds)  - \varepsilon.
    \]
    Therefore, for these $k$ we obtain
    \begin{multline*}
 \int_{\Xc} v(s)\;q(ds) + \left(\int_{\Xc} \bigg[v(s)- \int_{\Xc} v(s')\;q(ds')\bigg]_+^p \;q_k(ds)\right)^{{1}/{p}} \\
 \le
 \int_{\Xc} v(s)\;q_k(ds) +
 \left(\int_{\Xc} \bigg[v(s)- \int_{\Xc} v(s')\;q_k(ds')\bigg]_+^p \;q_k(ds)\right)^{{1}/{p}}
 + \varepsilon.
  \end{multline*}
  Taking the ``$\liminf$'' of both sides, and using the weak convergence of $q_k$ to $q$ and the lower semicontinuity of the functions integrated, we conclude that
 \begin{multline*}
  \int_{\Xc} v(s)\;q(ds) + \left(\int_{\Xc} \bigg[v(s)- \int_{\Xc} v(s')\;q(ds')\bigg]_+^p \;q(ds)\right)^{{1}/{p}}\\
  \le
 \liminf_{k\to \infty} \left\{ \int_{\Xc} v(s)\;q_k(ds) +
 \left(\int_{\Xc} \bigg[v(s)- \int_{\Xc} v(s')\;q_k(ds')\bigg]_+^p \;q_k(ds)\right)^{{1}/{p}} \right\}+ \varepsilon.
  \end{multline*}
 As $\varepsilon>0$ was arbitrary, the last relation proves the lower semicontinuity of $\sigma$ in the case when $\varkappa(x)\equiv 1$.
 The case of a continuous $\varkappa(x)\in [0,1]$ can be now easily analyzed by noticing that $\sigma(x,q,v)$ is a convex combination
 of the expected value, and the risk measure of the last displayed relation:
 \begin{multline*}
 \sigma(x,q,v) = \big( 1 -\varkappa(x)\big) \int_{\Xc} v(s)\;q(ds)\\{}  + \varkappa(x)
 \left\{ \int_{\Xc} v(s)\;q(ds)
  +
\left(\int_{\Xc} \bigg[v(s)- \int_{\Xc} v(s')\;q(ds')\bigg]_+^p \;q(ds)\right)^{{1}/{p}} \right\}.
 \end{multline*}
 As both components are lower semicontinuous in $(x,q)$, so is their sum.
      \end{proof}
    We can also verify
the weak continuity of the Average-Value-at-Risk transition risk mapping of Example \ref{e:avar-sigma}.
 To make the mapping Markovian, we assume that the parameter $\alpha$ depends on $x$ only, that is,
\begin{equation}
			\label{avar-lsc}
			\sigma(x,q,v) = \min_{\eta\in \R}
\left\{\eta + \frac{1}{\alpha(x)}\int_{\Xc} \big(  v(s)- \eta \big)_+\;q(ds)\right\},
\end{equation}
    For simplicity, we skip the subscript $t$ of $\sigma$ and $\alpha$.
    \begin{lemma}
    \label{l:avar-lsc}
    Suppose $\alpha(\cdot)$ is continuous and takes values in $[\alpha_{\min},\alpha_{\max}]\subset (0,1)$. Then for every continuous function $v$, the mapping $(x,q)\mapsto \sigma(x,q,v)$ in \eqref{avar-lsc} is continuous.
    \end{lemma}
    \begin{proof}
    Consider the function
    \begin{equation}
    \label{braces}
    (q,\eta) \mapsto \int_{\Xc} \big(  v(s)- \eta \big)_+\;q(ds).
    \end{equation}
    Suppose $q_k\to q$ weakly and  $\eta_k\to \eta$. We have
    \[
  \int_{\Xc} \big(  v(s)- \eta \big)_+\;q_k(ds) \le
   \int_{\Xc} \big(  v(s)- \eta_k \big)_+\;q_k(ds) + |\eta_k-\eta|.
    \]
    Taking the ``$\liminf$'' of both sides and using the weak convergence of $q_k$ to $q$
    and the lower semicontinuity of the functions integrated, we conclude that
    \[
    \int_{\Xc} \big(  v(s)- \eta \big)_+\;q(ds) \le \liminf_{k\to\infty}\int_{\Xc} \big(  v(s)- \eta \big)_+\;q_k(ds) 
    \le
    \liminf_{k\to\infty}\int_{\Xc} \big(  v(s)- \eta_k \big)_+\;q_k(ds).
    \]
    Thus the function \eqref{braces} is lower semicontinuous. It follows that the function
    being minimized with respect to $\eta$ in \eqref{avar-lsc} is jointly lower semicontinuous
    with respect to $(x,q,\eta)$. Since $q_k\to q$ weakly, the collection $\{q_k\}$ is \emph{tight}
    (Prohorov's theorem; see, \emph{e.g.}, \cite[Sec. 1.6]{billingsley2013convergence}).
    Since $v(\cdot)$ is continuous, the measures $q_k\circ v^{-1}$ are tight as well. Hence,
   a bounded interval $C\subset \R$ exists such that all $\alpha$-quantiles of all $q_k\circ v^{-1}$ and $q\circ v^{-1}$
    are contained in $C$, for all $\alpha \in [\alpha_{\min},\alpha_{\max}]$.
    Therefore, we can restrict $\eta$ to $C$ in \eqref{avar-lsc},
    without affecting the values of $\sigma(x_k,q_k,v)$ and $\sigma(x,q,v)$.
    By Berge theorem,
    the optimal value in \eqref{avar-lsc} is continuous (see, \emph{e.g.}, \cite[Thm. 1.4.16]{aubin2009set}).
     
    \end{proof}

\bibliographystyle{plain} 

\begin{thebibliography}{10}

\bibitem{arlotto2014markov}
A.~Arlotto, N.~Gans, and J.~M. Steele.
\newblock Markov decision problems where means bound variances.
\newblock {\em Operations Research}, 62(4):864--875, 2014.

\bibitem{ADEHK:2007}
P.~Artzner, F.~Delbaen, J.-M. Eber, D.~Heath, and H.~Ku.
\newblock Coherent multiperiod risk adjusted values and {B}ellman's principle.
\newblock {\em Annals of Operations Research}, 152:5--22, 2007.

\bibitem{aubin2009set}
J.-P. Aubin and H.~Frankowska.
\newblock {\em Set-valued analysis}.
\newblock Birkh\"auser, Boston, MA, 2009.

\bibitem{bauerle2013more}
N.~B{\"a}uerle and U.~Rieder.
\newblock More risk-sensitive {M}arkov decision processes.
\newblock {\em Mathematics of Operations Research}, 39(1):105--120, 2013.

\bibitem{bielecki1999risk}
T.~Bielecki, D.~Hern{\'a}ndez-Hern{\'a}ndez, and S.~R. Pliska.
\newblock Risk sensitive control of finite state {M}arkov chains in discrete
  time, with applications to portfolio management.
\newblock {\em Mathematical Methods of Operations Research}, 50(2):167--188,
  1999.

\bibitem{billingsley2013convergence}
P.~Billingsley.
\newblock {\em Convergence of {P}robability {M}easures}.
\newblock John Wiley \& Sons, 2013.

\bibitem{ccavus2014computational}
{\"O}.~{\c{C}}avus and A.~Ruszczy\'nski.
\newblock Computational methods for risk-averse undiscounted transient {M}arkov
  models.
\newblock {\em Operations Research}, 62(2):401--417, 2014.

\bibitem{ccavus2014transient}
{\"O}.~{\c{C}}avus and A.~Ruszczy\'nski.
\newblock Risk-averse control of undiscounted transient {M}arkov models.
\newblock {\em SIAM Journal on Control and Optimization}, 52(6):3935–--3966,
  2014.

\bibitem{ChenLiZhao2014}
Z.~Chen, G.~Li, and Y.~Zhao.
\newblock Time-consistent investment policies in {M}arkovian markets: a case of
  mean-variance analysis.
\newblock {\em J. Econom. Dynam. Control}, 40:293--316, 2014.

\bibitem{CDK:2006}
P.~Cheridito, F.~Delbaen, and M.~Kupper.
\newblock Dynamic monetary risk measures for bounded discrete-time processes.
\newblock {\em Electronic Journal of Probability}, 11:57--106, 2006.

\bibitem{cheridito2011composition}
P.~Cheridito and M.~Kupper.
\newblock Composition of time-consistent dynamic monetary risk measures in
  discrete time.
\newblock {\em International Journal of Theoretical and Applied Finance},
  14(01):137--162, 2011.

\bibitem{Coraluppi1999}
S.~P. Coraluppi and S.~I. Marcus.
\newblock Risk-sensitive and minimax control of discrete-time, finite-state
  {M}arkov decision processes.
\newblock {\em Automatica}, 35(2):301--309, 1999.

\bibitem{dai1998explicit}
P.~Dai~Pra, L.~Meneghini, and W.~J. Runggaldier.
\newblock Explicit solutions for multivariate, discrete-time control problems
  under uncertainty.
\newblock {\em Systems \& Control Letters}, 34(4):169--176, 1998.

\bibitem{Denardo1979}
E.~V. Denardo and U.~G. Rothblum.
\newblock Optimal stopping, exponential utility, and linear programming.
\newblock {\em Math. Programming}, 16(2):228--244, 1979.

\bibitem{DiMasi1999}
G.~B. Di~Masi and \L. Stettner.
\newblock Risk-sensitive control of discrete-time {M}arkov processes with
  infinite horizon.
\newblock {\em SIAM J. Control Optim.}, 38(1):61--78, 1999.

\bibitem{Filar1989}
J.~A. Filar, L.~C.~M. Kallenberg, and H.-M. Lee.
\newblock Variance-penalized {M}arkov decision processes.
\newblock {\em Math. Oper. Res.}, 14(1):147--161, 1989.

\bibitem{follmer2006convex}
H.~F{\"o}llmer and I.~Penner.
\newblock Convex risk measures and the dynamics of their penalty functions.
\newblock {\em Statistics \& Decisions}, 24(1/2006):61--96, 2006.

\bibitem{Howard1971}
R.~A. Howard and J.~E. Matheson.
\newblock Risk-sensitive {M}arkov decision processes.
\newblock {\em Management Sci.}, 18:356--369, 1971/72.

\bibitem{Jaquette1973}
S.~C. Jaquette.
\newblock Markov decision processes with a new optimality criterion: discrete
  time.
\newblock {\em Ann. Statist.}, 1:496--505, 1973.

\bibitem{Jaquette1975}
S.~C. Jaquette.
\newblock A utility criterion for {M}arkov decision processes.
\newblock {\em Management Sci.}, 23(1):43--49, 1975/76.

\bibitem{jaskiewicz2013persistently}
A.~Jaskiewicz, J.~Matkowski, and A.~S. Nowak.
\newblock Persistently optimal policies in stochastic dynamic programming with
  generalized discounting.
\newblock {\em Mathematics of Operations Research}, 38(1):108--121, 2013.

\bibitem{jobert2008valuations}
A.~Jobert and L.~C.~G. Rogers.
\newblock Valuations and dynamic convex risk measures.
\newblock {\em Mathematical Finance}, 18(1):1--22, 2008.

\bibitem{KloSch:2008}
S.~Kl{\"o}ppel and M.~Schweizer.
\newblock Dynamic indifference valuation via convex risk measures.
\newblock {\em Math. Finance}, 17(4):599--627, 2007.

\bibitem{koopmans1960stationary}
T.~C. Koopmans.
\newblock Stationary ordinal utility and impatience.
\newblock {\em Econometrica}, pages 287--309, 1960.

\bibitem{kupper2009representation}
M.~Kupper and W.~Schachermayer.
\newblock Representation results for law invariant time consistent functions.
\newblock {\em Mathematics and Financial Economics}, 2(3):189--210, 2009.

\bibitem{kuratowski1965general}
K.~Kuratowski and C.~Ryll-Nardzewski.
\newblock A general theorem on selectors.
\newblock {\em Bull. Acad. Polon. Sci. Ser. Sci. Math. Astronom. Phys},
  13(1):397--403, 1965.

\bibitem{Levitt2001}
S.~Levitt and A.~Ben-Israel.
\newblock On modeling risk in {M}arkov decision processes.
\newblock In {\em Optimization and related topics ({B}allarat/{M}elbourne,
  1999)}, volume~47 of {\em Appl. Optim.}, pages 27--40. Kluwer Acad. Publ.,
  Dordrecht, 2001.

\bibitem{lin2013dynamic}
K.~Lin and S.~I. Marcus.
\newblock Dynamic programming with non-convex risk-sensitive measures.
\newblock In {\em American Control Conference (ACC), 2013}, pages 6778--6783.
  IEEE, 2013.

\bibitem{Mannor2013}
S.~Mannor and J.~N. Tsitsiklis.
\newblock Algorithmic aspects of mean-variance optimization in {M}arkov
  decision processes.
\newblock {\em European J. Oper. Res.}, 231(3):645--653, 2013.

\bibitem{Marcus1997}
S.~I. Marcus, E.~Fern{\'a}ndez-Gaucherand, D.~Hern{\'a}ndez-Hern{\'a}ndez,
  S.~Coraluppi, and P.~Fard.
\newblock Risk sensitive {M}arkov decision processes.
\newblock In {\em Systems and control in the twenty-first century ({S}t.\
  {L}ouis, {MO}, 1996)}, volume~22 of {\em Progr. Systems Control Theory},
  pages 263--279. Birkh\"auser, Boston, MA, 1997.

\bibitem{ogryczak1999stochastic}
W.~Ogryczak and A.~Ruszczy{\'n}ski.
\newblock From stochastic dominance to mean-risk models: Semideviations as risk
  measures.
\newblock {\em European Journal of Operational Research}, 116(1):33--50, 1999.

\bibitem{ogryczak2001consistency}
W.~Ogryczak and A.~Ruszczy{\'n}ski.
\newblock On consistency of stochastic dominance and mean--semideviation
  models.
\newblock {\em Mathematical Programming}, 89(2):217--232, 2001.

\bibitem{ogryczak2002dual}
W.~Ogryczak and A.~Ruszczy\'nski.
\newblock Dual stochastic dominance and related mean risk models.
\newblock {\em SIAM Journal on Optimization}, 13:60--78, 2002.

\bibitem{PflRom:07}
G.Ch. Pflug and W.~R\"omisch.
\newblock {\em Modeling, Measuring and Managing Risk}.
\newblock World Scientific, Singapore, 2007.

\bibitem{Riedel:2004}
F.~Riedel.
\newblock Dynamic coherent risk measures.
\newblock {\em Stochastic Processes and Their Applications}, 112:185--200,
  2004.

\bibitem{rockuryas}
R.~T. Rockafellar and S.~Uryasev.
\newblock Optimization of conditional value at risk.
\newblock {\em The Journal of Risk}, 2:21--41, 2000.

\bibitem{rockafellar1998variational}
R~T. Rockafellar and R.~J.-B. Wets.
\newblock {\em Variational analysis}, volume 317.
\newblock Springer, Berlin, 1998.

\bibitem{roorda2005coherent}
B.~Roorda, J.~M. Schumacher, and J.~Engwerda.
\newblock Coherent acceptability measures in multiperiod models.
\newblock {\em Mathematical Finance}, 15(4):589--612, 2005.

\bibitem{runggaldier1998concepts}
W.~J. Runggaldier.
\newblock Concepts and methods for discrete and continuous time control under
  uncertainty.
\newblock {\em Insurance: Mathematics and Economics}, 22(1):25--39, 1998.

\bibitem{Ruszczynski2010Markov}
A.~Ruszczy{\'n}ski.
\newblock Risk-averse dynamic programming for {M}arkov decision processes.
\newblock {\em Math. Program.}, 125(2, Ser. B):235--261, 2010.

\bibitem{RuSh:2006b}
A.~Ruszczy\'{n}ski and A.~Shapiro.
\newblock Conditional risk mappings.
\newblock {\em Mathematics of Operations Research}, 31:544--561, 2006.

\bibitem{Scandolo:2003}
G.~Scandolo.
\newblock {\em Risk Measures in a Dynamic Setting}.
\newblock PhD thesis, Universit\`{a} degli Studi di Milano, Milan, Italy, 2003.

\bibitem{shapiro2012time}
A.~Shapiro.
\newblock Time consistency of dynamic risk measures.
\newblock {\em Operations Research Letters}, 40(6):436--439, 2012.

\bibitem{shen2013risk}
Y.~Shen, W.~Stannat, and K.~Obermayer.
\newblock Risk-sensitive {M}arkov control processes.
\newblock {\em SIAM Journal on Control and Optimization}, 51(5):3652--3672,
  2013.

\bibitem{weber2006distribution}
S.~Weber.
\newblock Distribution-invariant risk measures, information, and dynamic
  consistency.
\newblock {\em Mathematical Finance}, 16(2):419--441, 2006.

\bibitem{White1988}
D.~J. White.
\newblock Mean, variance, and probabilistic criteria in finite {M}arkov
  decision processes: a review.
\newblock {\em J. Optim. Theory Appl.}, 56(1):1--29, 1988.

\end{thebibliography}

\end{document}